\documentclass[10pt]{amsart}
\usepackage{amssymb,amsfonts,amsthm,amsmath,graphicx}
\usepackage{hyperref}
\usepackage{cite}

\author{Benjamin Steinhurst}
\title{Uniqueness of Locally Symmetric Brownian Motion on Laakso Spaces}

\newcommand{\C}[1]{\ensuremath{\mathcal{#1}}}
\newcommand{\B}[1]{\ensuremath{\mathbb{#1}}}
\newcommand{\F}[1]{\ensuremath{\mathfrak{#1}}}

\newtheorem{theorem}{Theorem}[section]
\newtheorem{cor}{Corollary}[section]
\newtheorem{lemma}{Lemma}[section]
\newtheorem{prop}{Proposition}[section]

\newtheorem{definition}{Definition}[section]

\numberwithin{equation}{section}
\date{\today\thanks{Research supported in part by the NSF grant DMS-0505622.}}
\begin{document}

\begin{abstract}
We consider the spaces introduced by Laakso in 2000 and, building on the work of Barlow, Bass, Kumagai, and Teplyaev, prove the existence and uniqueness of a local symmetry invariant diffusion via heat kernel estimates. This work also builds upon works of Cheeger, Barlow and Bass, as well as the author.

\end{abstract}

\maketitle

\begin{center}{\small
Contact:\\
steinhurst@math.cornell.edu\\
Department of Mathematics\\
Cornell University\\
Ithaca~NY~14850~USA}\end{center}

\section{Introduction}
This paper seeks to take the methods used in \cite{BBKT2010} to show the uniqueness of locally symmetric Brownian motion on the Sierpinski carpet and apply them to Brownian motion on Laakso spaces, simplifying them where the geometry of Laakso spaces allow. These spaces were originally defined by Laakso in \cite{Laakso2000}, given an alternate construction in \cite{Steinhurst2010,RomeoSteinhurst2009} based on \cite{BarlowEvans2004}, and the Laplacian analyzed in \cite{Steinhurst2010,RomeoSteinhurst2009}.

In Section \ref{sec:laakso} the family of Laakso spaces is defined and the local symmetries defined. Then Section in \ref{sec:processes} the probabilistic arguments needed to prove the Elliptic Harnack Inequality (EHI) for harmonic functions on Laakso spaces are proven concluding with the proof of the EHI. From the EHI heat kernel estimates are proven in Section \ref{sec:hkest}. Section \ref{sec:uniq} contains Theorem \ref{thm:onedim}, the main result of this paper, that local symmetry invariant Dirichlet forms are equal up to scalar multiple. Finally in Section \ref{sec:MGUG-Lap} the existence of a local symmetry invariant Dirichlet form is shown by taking the Dirichlet form analyzed in \cite{Steinhurst2010} which is then shown to be invariant. The outline of this argument, and indeed many of the proofs are based on those in \cite{BBKT2010} in which the same result was shown for generalized Sierpinski carpets.  

{\bf Acknowledgments:} The author thanks Alexander Teplyaev, Jun Kigami and Naotaka Kajino for their helpful suggestions.

\subsection{Generalities of Dirichlet Forms}\label{ssec:generalities}

Let $(X,m,d)$ be a metric measure space. We denote a regular Dirichlet form on $L^{2}(X,m)$ by $(\C{E},\C{F})$. By this we mean that \C{E} is a bilinear, closed, symmetric form with the Markov property with domain $\C{F} \subset L^{2}(X,m)$ such that $C_0(X) \subset \C{F}$ densely and $\C{F} \cap C_0(X) \subset C_0(X)$ densely in the respective topologies. 

The semi-group of operators $T_t$ on $L^{2}(X,m)$ associated to \C{E} is the semi-group whose strong derivative at zero is the same self-adjoint operator as the generator of \C{E}. If $A$ is the non-positive self-adjoint operator that is the derivative of $T_t$ then we have
\begin{equation}
	\C{E}(u,v) = \langle \sqrt{-A}u,\sqrt{-A}v) \rangle \hspace{1cm} \text{and} \hspace{1cm} T_t(u) = e^{tA}(u),
\end{equation}
for all $u \in \C{F}$. There is a large literature on Dirichlet forms and specifically their connection with self-adjoint operators resolvents, semi-groups of operators, and Hunt processes. The standard reference used in this paper is \cite{FOT1994}. 

\begin{definition} Assume that $m(X) < \infty$. Let $(\C{E},\C{F})$ be a regular Dirichlet form. 
\begin{enumerate}
	\item The form \C{E} is local if for any $u,v \in \C{F}$ with disjoint compact support $\C{E}(u,v) = 0$.
	\item The form \C{E} is conservative if the associated semi-group satisfies $T_t 1 = 1\ a.e.$ for all $t > 0$.
	\item The form \C{E} is irreducible if the only set of positive measure which is invariant under $T_t$ is $X$. 
\end{enumerate}
\end{definition}

The following type of linear combination is going to be of reoccuring interest throughout the paper. Specifically in Theorems \ref{thm:hilbertbound} and \ref{thm:onedim} where it is shown that a particular family of Dirichlet forms closed under this type of linear combination is one-dimensional.

\begin{theorem}[\cite{BBKT2010}]\label{thm:combo}
Suppose that $(\C{A},\C{F})$ and $(\C{B},\C{F})$ are local, regular, conservative, irreducible Dirichlet forms on $L^{2}(F,m)$ and that 
\begin{equation}
	\C{A}(u,u) \le \C{B}(u,u)
\end{equation}
for all $u \in \C{F}$. For any  $\delta >0$, set $\C{E} = (1+\delta)\C{B} - \C{A}$. Then $(\C{E},\C{F})$ is a local, regular, conservative, irreducible Dirichlet form on $L^{2}(F,m)$.
\end{theorem}

\emph{Comment on proof:}  All of the properties except the conservativity and irreducibility follow immediately from the definitions. That \C{E} is conservative follows from \C{A} and \C{B} being conservative and the compactness of $L$. The argument for irreducibility is found in the proof of Theorem 2.1 in \cite{BBKT2010}.

\section{Laakso Spaces}\label{sec:laakso}
Laakso spaces were introduced in \cite{Laakso2000} in order to exhibit metric measure spaces of every Hausdorff dimension greater than one that all support a Poincar\'e inequality.  Because of the existence of an entire family of spaces it is possible to speak of those properties which hold for all Laakso spaces and those properties which depend on the member of the family that is under consideration. For example the metric on the Cantor sets used in the construction will depend on the desired Hausdorff dimension while the essentially one dimensional behavior of the Laplacian that will be defined in Section \ref{sec:MGUG-Lap} does not. To be able to define Laakso spaces we first give definitions for several objects. We then prove a few properties of Laakso spaces.

\subsection{Construction of Laakso Spaces}\label{ssec:DefLaakso}
Let $K$ be a Cantor set with its elements denoted by infinite sequences of $0$'s and $1$'s. For an integer $j \ge 2$ and a sequence $j_i \in \{j, j+1\}$ define
\begin{equation}
	d_N = \prod_{j=1}^{N} j_i \hspace{1in} L_N = \left\{ \frac{i}{d_n} \right\}_{i=1}^{d_N-1}.
\end{equation}
Shortly $d_N$ will be used in relation to the diameter of the cells to be defined below while $L_N$ will indicate where identifications will be made. 

\begin{definition}\label{def:trans} We set notation for three functions on Cantor sets.
\begin{itemize}
	\item Let $T_n$ act on elements of a Cantor set by transposing the $n$'th digit in the infinite sequence representation. The operator $T_n$ can be extended to products of Cantor set $K^{k}$ by letting $T_n$ operate on each copy of $K$ separately. 
	\item Let $\sigma:K \rightarrow K$ be the shift function mapping so that for $S \in \C{S}_1$ $\sigma(S) = K$.
	\item Let $\psi_a:K \rightarrow K$ be the contraction map sending $K$ to the subset of elements starting with the word $a$. 
\end{itemize}
\end{definition}

\begin{definition}\label{def:laakso}
Set $\sim_{n}$ to be an equivalence relation on $I \times K^{k}$ by identifying points $(x,w), (x,T_n w) \in I \times K^{k}$ for $x \in L_n \setminus L_{n-1}$. Denote by $\sim$ the union of the equivalence relations, $\sim_n$. Then $\sim$ is again an equivalence relation because the sets $L_n \setminus L_{n-1}$ are disjoint over $n$. A Laakso space is defined as $L= I \times K^{k} / \sim$. Denote the map $I \times K^{k} \mapsto I \times K^{k} / \sim = L$ by $\iota$, it is called the wormhole map. 
\end{definition}

The data determining a Laakso space are $\{j_i\}$ and $k$. In \cite{RomeoSteinhurst2009,BegueEtAl2009} the spectral properties of a Laplacian are given explicitly using only this data.

\begin{prop}\label{prop:laaksoproperties}
Any Laakso space is metrizable with a geodesic metric. Relative to the geodesic metric the Hausdorff dimension of $L$ is
\begin{equation}
	Q = dim_H(L) = \lim_{l \rightarrow \infty} \frac{\log(2^{lk}d_l)}{\log(d_l^{-1})}
\end{equation}
whenever this limit exists. If $K$ is given a Bernoulli measure and and $I$ the Lebesgue measure, then $L$ inherits a measure from $I \times K^{k}$. The inherited measure and the Hausdorff measure relative to the geodesic metric are equivalent. 
\end{prop}

\begin{proof} The first assertion is in \cite{Laakso2000}. The second claim follows from \cite{Igudesman2003} by noting that any single-valued branch of $\iota^{-1}(L )$ is a lacunary self-similar set as a subset of $\B{R}^{1+k}$ with the pull-back metric and  contraction factor $j_i^{-1}$ at each level. A lacunary self-similar set is one where the contraction ratios at each level have a geometric mean even if they are not periodic. The formula is an immediate consequence of Theorem 2 in \cite{Igudesman2003}. 
If $\{U\}$ is an open cover of a branch of $\iota^{-1}(L)$ then $\{(1+\delta)U\}$ is an open cover of $\iota^{-1}(L)$ for all $\delta > 0$. So $\iota^{-1}(L)$has the same Hausdorff dimension as any of the branches. 
In this context $\iota$ is a $1-$bi-Lipschitz map from $\iota^{-1}(L) \rightarrow L$.
\end{proof}

To denote the distance between two points $x,y \in L$  we will use the notation $|x-y|$ even though there is no additive group structure on $L$. 

Many of the proofs in this paper follow closely the arguments in \cite{BBKT2010,BarlowBass1999}. The uniqueness result on locally symmetric Brownian motions on Laakso spaces have been shown to hold for generalized Sierpinski carpets in \cite{BBKT2010}. Many of the proofs used the specific geometry of generalized Sierpinski carpets to estimate the probability of processes hitting particular sets. These estimates were concerned with two types of trajectories: corner moves and slide moves. The corner move was the cause of most of the technical difficulties. Laakso spaces can be thought of as having slide moves, where a process moves from one piece of the boundary to another, adjacent piece of the boundary of a cell. In Section \ref{sec:processes} these moves will be mentioned again in describing which arguments carry over from the previous literature. 

\subsection{Symmetries of Laakso Spaces}\label{ssec:SymLaakso}
We will define a folding map that will be used to describe the symmetries of the Dirichlet forms on Laakso spaces. Before we define the folding map we first need a cell structure.

\begin{definition}\label{def:cellstructure}
Let $L_N=\{x_n\}_{n=1}^{d_N-1}$ to be the wormhole locations for all depths less than or equal to $N$ as in the Definition \ref{def:laakso} listed in increasing order. Let $K_j$, for $j = 1, \cdots, 2^{Nk}$, be the depth $N$ cells of $K^{k}$. Then a depth $N$ cell of $L$ is $S = \iota([x_i,x_{i+1}] \times K_j)$. The set of all such $N$-cells is called $\C{S}_N$.
\end{definition}

\begin{definition}\label{def:folding}
For $n \ge 0$ let $S \in \C{S}_n$ then $\iota^{-1}(S) = [x_i,x_{i+1}] \times K_a.$ Then taking $\overline{\varphi}_0(x) = \max\{0, \min\{ x, 1-x\}\}$extended periodically to all of \B{R} \cite[Definition 2.12]{BBKT2010} define $\varphi_S:[0,1] \rightarrow [x_i,x_{i+1}]$  by $\varphi_S(z) = x_i d_n^{-1} \overline{\varphi}_0(d_n(x-x_i))$. Define $K_S:K \rightarrow K_a$ by $K_S = \psi_a \circ \sigma^{n}$. Let $\phi_S$ be defined by 
\begin{equation}
	\phi_S = \iota \circ (\varphi_S,K_S) \circ \iota^{-1}.
\end{equation}
\end{definition}

\begin{lemma}\label{lem:phiprop}
\ 
\begin{enumerate}
	\item For every $n \ge 0$ and $S,S' \in \C{S}_n$ $\phi_S: S' \rightarrow S$ is an isometry.
	\item For every $n \ge 0$ and $S_1,S_2 \in \C{S}_n$ $\phi_{S_1}\circ \phi_{S_2} = \phi_{S_1}.$
	\item For every $n \ge 0$ and $x,y \in L$. If there exists $S_1 \in \C{S}_n$ such that $\phi_{S_1}(x) = \phi_{S_1}(y)$, then $\phi_{S}(x) = \phi_{S}(y)$ for every $S \in \C{S}_n$. 
	\item For every $n \ge 0$ let $S \in \C{S}_n$ and $S' \in \C{S}_{n+1}$. If $x,y \in L $ and $\phi_{S}(x)=\phi_{S}(y)$ then $\phi_{S'}(x)=\phi_{S'}(y).$
\end{enumerate}
\end{lemma}

\begin{proof} 
\ 
\begin{enumerate}
	\item Let $S \in \C{S}_n$ and $\iota^{-1}(S) = [a,b] \times K_a$ where $|a|=n$. Then $K_S(K_a) = K_a$ and if $[a,b]$ is the $n-$cell used to define $\overline{\varphi}_0$,$\overline{\varphi}_0([a,b]) = id([a,b]) = [a,b]$. Then $\iota \circ (id, id) \circ \iota^{-1}(S) = S$, since $\iota \circ \iota^{-1} = id$. If $S' \neq S$ then $\phi_S(S') = \iota \circ (\overline{\varphi}_0, K_S) \circ \iota^{-1}(S') =  \iota \circ (\overline{\varphi}_0, K_S)([c,d] \times K_b)$, where $|b| = n$, which equals $\iota([a,b]\times K_a) = S$. Because $(\overline{\varphi}_0, K_S)$ is an isometry on $I \times K^{k}$ the composition is also an isometry.
	\item 
	\begin{eqnarray}
		\phi_{S_1} \circ \phi_{S_1} &=& \left(\iota \circ ( \overline{\varphi}_{S_1}, K_{S_1}) \circ \iota^{-1}\right) \circ \left(\iota \circ ( \overline{\varphi}_{S_2}, K_{S_2}) \circ \iota^{-1}\right)\\
		&=& \iota \circ (\overline{\varphi}_{S_1},K_{S_1}) \circ (\overline{\varphi}_{S_2},K_{S_2}) \circ \iota^{-1} \\
		&=& \iota \circ (\overline{\varphi}_{S_1} \circ \overline{\varphi}_{S_2}, K_{S_1} \circ K_{S_2} ) \circ \iota^{-1}\\
		&=& \iota \circ (\overline{\varphi}_{S_1}, K_{S_1}) \circ \iota^{-1}\\
		&=& \phi_{S_1}
	\end{eqnarray}
	\item Let $x,f \in L$, and $S_1 \in \C{S}_n$ such that $\phi_{S_1}(x) = \phi_{S_1}(y)$ and $S_2 \in \C{S}_n$. Then $\phi_{S_2}(\phi_{S_{1}}(x)) = \phi_{S_2}(\phi_{S_1}(y))$ implies by part b that $\phi_{S_2}(x) = \phi_{S_2}(y)$ for all $S_2 \in \C{S}_n$. 
	\item Let $S \in \C{S}_n$, $S' \in \C{S}_{n+1}$, and $x,y \in L$ such that $\phi_S(x) = \phi_S(y)$. Then we are done if $\phi_{S'} \circ \phi_S = \phi_{S'}$. However both $\overline{\varphi}_0$ and $K_S$ have this property and the conjugation by $\iota$ does not disturb it. 
\end{enumerate}
\end{proof}

For fixed $S \in \C{S}_n$ we define restriction and unfolding operators. For $f$ defined on $L$ the restriction operator acts by $R_S f = \left. f \right|_S$ and for $g$ defined on $S$ the unfolding operator acts by $U_S g = g \circ \phi_S$.

\subsection{Invariant Dirichlet Forms}\label{ssec:invariantDF}
Let $L$ be a Laakso space then it has a cell structure $\{\C{S}_n\}_{n=0}^{\infty}$ as defined above. Set $m_L^{n} = \#\C{S}_n$ to be the number of $n^{th}$ level cells in $L$. It is worth noting that $m_L^{n}$ is not genreally $(m_L)^{n}$. This is only true when $j_i$ is a constant sequence.  Let $(\C{E},\C{F})$ be a local, regular Dirichlet form on $L^{2}(L,\mu)$. Let $S \in \C{S}_n$, set
\begin{equation}
	\C{E}^{S}(g,g) = \frac{1}{m_L^{n}}\C{E}(U_Sg,U_Sg).
\end{equation}
and define the domain of $\C{E}^{S}$ to be $\C{F}^{S} = \{ g: S \mapsto \B{R},\ U_Sg \in \C{F} \}$. We write $\mu_S = \mu |_S.$

\begin{definition}\label{def:invar}
Let $(\C{E},\C{F})$ be a Dirichlet form on $L^{2}(L,\mu)$.  We say that \C{E} is an $L-$invariant Dirichlet form or that \C{E} is invariant with respect to all the local symmetries of $L$ if the following items hold:
\begin{enumerate}
	\item If $S \in \C{S}_n$, then $U_SR_Sf \in \C{F}$ (i.e. $R_Sf \in \C{F}^{S}$) for any $f \in \C{F}$.
	\item Let $n \ge 0$ and $S_1,S_2$ be any two elements of $\C{S}_n$, and let $\Phi$ be any isometry of $\B{R}^{1+k}$ such that $\tilde{\Phi} = \iota \circ \Phi \circ \iota^{-1}$ maps $S_1$ onto $S_2$. (We allow $S_1 = S_2$). If $f \in \C{F}^{S_2}$, then $f \circ \tilde{\Phi} \in \C{F}^{S_1}$ and
	\begin{equation}
		\C{E}^{S_1}(f \circ \tilde{\Phi}, f \circ \tilde{\Phi}) = \C{E}^{S_2}(f,f).
	\end{equation}
	\item For all $f \in \C{F}$
	\begin{equation}
		\C{E}(f,f) = \sum_{S \in \C{S}_n} \C{E}^{S}(R_Sf,R_Sf).
	\end{equation}
\end{enumerate}

Let  $\F{C}$ be the family of all $L-$invariant, non-zero, local, regular, conservative Dirichlet forms. 
\end{definition}

\begin{definition}\label{def:theta}
For a fixed $n \ge 0$ and any $f \in \C{F}$ define
\begin{equation}
	\Theta f = \frac{1}{m_L^{n}} \sum_{S \in \C{S}_n} U_SR_S f.
\end{equation}
\end{definition}

It is straight forward to check that $\Theta^{2} = \Theta$  and that it is a bounded operator on $C(L)$, and a bounded self-adjoint operator on $L^{2}(L,\mu)$, and $\C{F}$. 

\begin{prop}\label{prop:commute}
Let \C{E} be a local regular Dirichlet form on $L$, $T_t$ the semigroup, and $U_SR_Sf \in \C{F}$ whenever $S \in \C{S}_n$ and $f \in \C{F}$. Then the following are equivalent:
\begin{enumerate}
	\item For all $f \in \C{F}$, we have $\C{E}(f,f) = \sum_{S \in \C{S}_n} \C{E}^{S}(R_Sf,R_Sf)$;
	\item For all $f,g \in \C{F}$
	\begin{equation}
		\C{E}(\Theta f,g) = \C{E}(f,\Theta g);
	\end{equation}
	\item For all $f \in L^{2}(L,\mu)$ then $T_t\Theta f = \Theta T_t f$ a.e. where $t \ge 0$.
	\item Let $A$ be the infinitesimal generator of $T_t$, then for all $f \in Dom(A)$ we have $A\Theta f = \Theta Af$.
\end{enumerate}
\end{prop}

\begin{proof} This is Proposition 2.21 in \cite{BBKT2010}, the proof is the same as well in light of Definition \ref{def:invar} above corresponding to Definition 2.15 in \cite{BBKT2010}. The connection between Dirichlet forms, semi-groups of operators, and infinitesimal generators is described in \cite{FOT1994}.
\end{proof}

\begin{prop}\label{prop:closedcombo}
The family \F{C} is closed under the linear combinations described in Theorem \ref{thm:combo}.
\end{prop}

\begin{proof}
Theorem \ref{thm:combo} gives that for any $\C{A},\C{B} \in \F{C}$ with common domain and $\C{A} \le \C{B}$ that $\C{E} = (1+\delta)\C{B}- \C{A}$ is a local, conservative, irreducible, regular Dirichlet form for all $\delta > 0$. To see that \C{E} is also invariant all that is necessary is to write out the conditions in Definition \ref{def:invar} as applying to $(1+\delta)\C{B} - \C{A}$ and check that the conditions hold. Also $\C{E}$ is non-zero as long as $\delta > 0$. 
\end{proof}

This type of combination is only valid for Dirichlet forms with the same domain and the possibility of having different domains appear for elements of \F{C} has not yet been excluded.

\section{Processes on Laakso Spaces}\label{sec:processes}
The main goal of this paper is to establish heat kernel estimates for operators in \F{C}.  Following the arguments of \cite{BBKT2010} we use a probabalistic approach to arrive at an elliptic Harnack inequality. 

Let $(X_t; t\ge 0)$ be the diffusion associated with a Dirchlet form that is $L$-invariant (Definition \ref{def:invar}) on $L$ with laws $\B{P}^{z}$ for $z \in L$. 

\subsection{The Reflected Process}\label{ssec:reflected}
We condense many of the results from \cite{BBKT2010} which describe the properties of the reflected processes, any skipped details can be found there.

\begin{definition}\label{def:reflectedproc}
Let $X_t$ be a diffiusion on $L$ that is associated to a $L-$invariant Dirichlet form. Then for any choice $S \in \C{S}_n$ of cells along with its folding function, $\phi_S$ we define the reflected process in S by 
\begin{equation}
	Z_t^{S} = Z_t = \phi_S(X_t).
\end{equation}
Where the superscript $S$ in $Z^{S}_t$ is omitted when no confusion can occur. 
\end{definition}

\begin{theorem}\label{thm:foldedprocess}
Let $S \in \C{S}_n$. Then $Z$ is a $\mu_S-$symmetric Markov process with Dirichlet form $(\C{E}^{S},\C{F}^{S})$, and semi-group $T_t^{Z}f = R_sT_tU_Sf$. Write $\tilde{\B{P}}^{y}$ for the laws of $Z$; these are defined for $y \in S \setminus \C{N}_2^{Z}$, where $\C{N}_2^{Z}$ is a properly exceptional set for $Z$. There exists a properly exceptional set $\C{N}_2$ for $X$ such that for any Borel set $A \in L$,
\begin{equation}
	\tilde{\B{P}}^{\phi_S(x)}(Z-t \in A) = \B{P}^{x}(X_t \in \phi_S^{-1}(A))
\end{equation}
for $x \in L \setminus \C{N}_2$.
\end{theorem}

\begin{proof}
To ease notation we drop the subscript $S$ from $\phi_S$. Our first claim is the existence of the properly exceptional $\C{N}_2$ set for $X$ such that
\begin{equation}
	\B{P}^{x}(X_t \in \phi^{-1}(A)) = T_t1_{\phi^{-1}(A)}(x) = T_t 1_{\phi^{-1}(A)}(y) = \B{P}^{x}(X_t \in \phi^{-1}(A))
\end{equation}
for all Borel $A$, $x,y \in L \setminus \C{N}_2$ such that $\phi(x) = \phi(y)$. It is only necessary to prove this relation on a countable base $(A_m)$ of the Borel $\sigma-$field. Note that $1_{\phi^{-1}(A_m)} = U_S1_{A_m}$ so the claim reduces to showing that 
\begin{equation}
	T_tU_S1_{A_m}(x) = T_tU_S1_{A_m}(y)\label{eq:reflect}
\end{equation}
for $x,y \in L \setminus \C{N}_2$ and $\phi(x) = \phi(y)$. But $U_S 1_{A_m}$ is invariant under $\Theta$ and by Proposition \ref{prop:commute} we have 
\begin{equation}
	\Theta T_t U_S f = T_t \Theta U_S f = T_t U_S f\ q.e.
\end{equation}
Thus there is a properly exceptional set $\C{N}_{2,m}$ such that (\ref{eq:reflect}) holds off of $\C{N}_{2,m}$. Take $\C{N}_2 = \bigcup_m \C{N}_{2,m}$ and $\C{N}_2^{Z} = \phi(\C{N}_2)$. Theorem 10.13 of \cite{Dynkin1965} shows that $Z$ is a Markov process and that the semi-groups are related by $T^{Z}_tf = R_ST_t(U_Sf)$. Then the last statement of the theorem to show is the symmetry of the process.

The proof of symmetry relies on two facts and a calculation. First is (\ref{eq:reflect}). The second is that $U_SR_ST_tU_Sf = T_tU_Sf$. The calculation is to write out both $\langle T_t^{Z} f, g \rangle_S$ and $\langle  f, T_t^{Z}g \rangle_S$, using the definition of $T_t^{Z}$, as $m_L^{-n}\langle T_tT_Sf,U_Sg \rangle$ and use the symmetry of $T_t$ to show symmetry of $T_t^{Z}$. Identifying the Dirichlet form follows by using $T_t^{Z}$ to approximate $\C{E}^{Z}$ and comparing to the definition of $\C{E}^{S}$.
\end{proof}

\begin{lemma}
Let $S,S' \in \C{S}_n$, and $\Phi$ be an isometry from $S$ to $S'$. Then for some properly exceptional set $\C{N}$ if $x \in S \setminus \C{N}$,
\begin{equation}
	\B{P}^{x}(\Phi(Z) \in \cdot ) = \B{P}^{\Phi(x)}(Z \in \cdot).
\end{equation}
\end{lemma}

\begin{proof}
This follows from Theorem \ref{thm:foldedprocess}, the definition of $L-$invariance, and the equivalence of processes which have the same Dirichlet forms upto sets of capacity zero.
\end{proof}

\begin{definition}
For Borel $D \subset L$ let 
\begin{equation}
	E_D = \{ x \in L :\ \B{P}^{x}(\tau_D = \infty) = 0 \}.
\end{equation}
Where $\tau_D$ is the exit time of $D$. 
\end{definition}

\begin{definition}
It will be useful to have a notion of ``half-face'' available. Let $S$ be a cell in $L$. Then it is of the form $\iota(K_a \times [x_1,x_2])$ where $a$ is a word of length $n$ and $x_1,x_2$ are adjacent in $L_n \cup \{0,1\}$. Then each subset of $L$ of the following form is called a half-face:
\begin{equation}
	 A_{i,k} = \{q \in L :\ \iota(x_i,w) = q,\ w \in K_{ak},\ k \in \{0,1\}^{k}\}.
\end{equation}
\end{definition}

\begin{lemma}\label{lem:hitprob}
Let $A_0,A_1$ be two half-faces of a cell $S \in \C{S}_n$ and $S_*$ the union of the elements of $\C{S}_n$ that contain $A_0$. Set $ \tau = \tau_{S_*}$. There exists a constant  $q_1>0$ such depending only on $L$ such that if $x \in A_0 \cap E_{S_*}$ and $T_0 \le \tau$ is a finite $(\C{F}^{Z}_t)$ stopping time, then
\begin{equation}
	\B{P}^{x}(X_{T_0} \in S | \C{F}_{T_0}^{Z}) \ge q_1.
\end{equation}
Moreover exists $q_0 >0$ such that
\begin{equation}
	\B{P}^{x}(T_{A_1}^{X} \le \tau) \ge q_0q_1.
\end{equation}
\end{lemma}

\begin{proof}
Because $X$ is a diffusion and $T_0 \le \tau$ we have that $X_{T_0} \in S_*$. From the geometry of $L$, $S_*$ contains only $2^{k}$ cells so $q_1 = 2^{-k}$ which is only dependent on $L$. 

Let $z_{\pm} \in L_{n+1}$ be then adjacent elements to the $x_i$ in the definition of the half-face $A_0$. Then almost any path of $X_t$ starting in $E_{S_*}$ will hit either $\iota(z_-,K^{k})$ or $\iota(z_+,K^{k})$ where $K$ is the product of Cantor sets in the definition of $L$. It has already been shown that with proability $q_1$ this hit occurs in $S \subset S_*$. In this case there are $2^{k+1}$ neighboring half-faces of level $n+1$, one of which comprises $A_1$. By the symmetry of the process $X$ the probability that the next one of these half-faces hit is equal. Hence  $q_0 \ge \frac{1}{2^{k+1}}$.
\end{proof}

\subsection{Coupling}\label{ssec:coupling}

Coupling two locally invariant processes will provide a key step in proving the elliptic Harnack inequality. We state the following result used to produce coupled processes in a manner that respects the local symmetry of the Laakso spaces.

\begin{lemma}[\cite{BBKT2010}]\label{lem:coupling}
Let $(\Omega,\C{F},\B{P})$ be a probability space. Let $X$ and $Z$ be random variables taking values in separable metric spaces $E_1$ and $E_2$, respectively, each furnished with the Borel $\sigma-$field. Then there exists $F:E_2 \times [0,1] \rightarrow E_1$ that is jointly measurable such that if $U$ is a random variable whose distribution is uniform on $[0,1]$ which is independent of $Z$ and $\tilde{X} = F(Z,U)$, then $(X,Z)$ and $(\tilde{X},Z)$ have the same law.
\end{lemma}

\begin{lemma}\label{lem:coupled}
Let $x_1,x_2 \in L$ where $x_i \in S_i \in \C{S}_n$, and let $\Phi = \phi_{S_1} |_{S_2}$. Then there exists a probability space $(\Omega,\C{F},\B{P})$ carrying processes $X_i, i=1,2$ and $Z$ with the following properties.
\begin{enumerate}
	\item Each $X_i$ is an $\C{E}-$diffusion started at $x_i$.
	\item $Z = \phi_{S_2}(X_2) = \Phi \circ \phi_{S_1}(X_1)$.
	\item $X_1$ and $X_2$ are conditionally independent given $Z$.
\end{enumerate}
\end{lemma}

\begin{proof} Let $Y_1, Y_2$ be diffusions corresponding to the Dirichlet form \C{E} that are equal in law and started at $x_1,x_2$ respectively.  Set $Z_i = \Phi \circ \phi_{S_i}(Y_i)$ for $i=1,2$. The Dirichlet form for $\phi_{S_i}(Y)$ is $\C{E}^{S_i}$ and the $Z_i$ have the same starting point they also are equal in law. Then Lemma \ref{lem:coupling} can be used to find functions $F_i$ such that $(F_i(Z_i,U,Z_i)$ are equal in law to $(Y_i,Z_i)$ for $i=1,2$, if $U$ is an independent uniform random variable on $[0,1]$.

On a probability space supporting a process $Z$ with the same law as the $Z_i$ and two independent random variables $U_1,U_2$ independent of $Z$ which are uniform on $[0,1]$ will be $X_i = F_i(Z,U_i),\ i=1,2$ that will satisfy the three properties in the statement of this lemma.

Since the $X_i$ is equal in law to $F_i(Z_i,U_i)$, which are equal in law to $Y_i,\ i=1,2$, this establishes (1). Similarly $(X_i,Z)$ are equal in law to $(F(Z_i,U_i),Z_i)$, which are equal in law to $(Y_i,Z_i)$. Because $Z_i = \Phi \circ \phi_{S_i}(Y_i)$ it follows from the equality in law that $Z=\Phi \circ \phi_{S_1}(Y_1) = \Phi \circ \phi_{S_2}(Y_2).$ This is (b).

As $X_i = F_i(Z,U_i)$ for $i=1,2$ and $Z,U_1$, and $U_2$ are independent, (3) is immediate.
\end{proof}

Given a pair of $\C{E}-$diffusions $X_1(t)$ and $X_2(t)$ define the \emph{coupling time}
\begin{equation}
	T_C(X_1,X_2) = \inf \{ t\ge 0: X_1(t) = X_2(t)\}.
\end{equation}
The coupling time of two diffusions is simply the first time they take a common value and then evolve together from that value.

\begin{theorem}\label{thm:coupling}
Let $r>0$, $\epsilon >0$, and $r' = r/(j+1)^{2}$. There exist constants $q$ and $\delta$, depending only on $L$ such that the following hold:
\begin{enumerate}
	\item Supose $x_1,x_2 \in L$ with $|x_1-x_2| < r'$ and $x_1 \sim_m x_2$ for $m \ge 1$. There exist $\C{E}-$diffusions $X_i(t), i=1,2,$ with $X_i(0)=x_i$, such that with 
	\begin{equation}
		\tau_i = \inf\{t \ge 0: X_i(t) \not\in B(x_i,r)\},
	\end{equation}
	We have
	\begin{equation}
		\B{P}(T_C(X_1,X_2) < \tau_1 \wedge \tau_2) > q.
	\end{equation}
	\item If in addition $|x_1-x_2| < \delta r$ and $x_1 \sim_m x_2$ for some $m \ge 1$ then 
	\begin{equation}
		\B{P}(T_C(X_1,X_2) < \tau_1 \wedge \tau_2) > 1-\epsilon.
	\end{equation}
\end{enumerate}
\end{theorem}

\begin{proof}
This proof follows from Lemmas \ref{lem:hitprob} and  \ref{lem:coupled} as in Theorem 3.25 in \cite{BarlowBass1999} since the cell structure of a Laakso space behaves similarly to the cell structure of a generalized Sierpinski carpet without the so called ``corner moves.'' In this theorem, the constant $q$ depends only $k$ and $j$, that is only on the Laakso space and not on $m$. 
\end{proof}

\subsection{Elliptic Harnack Inequality}\label{ssec:ehi}
Let $D$ be a relatively open subset of $L$ and $D'$ a relatively open subset of $D$. Let $X_t$ be a Markov process and $\tau_{D'}$ be the stopping time of $X_t$ leaving $D'$. If $h(X_{\tau_{D'}})$ is a martingale under $\B{P}^{x}$ for quasi every $x \in D'$ then $h$ is probabilistically harmonic with respect to $X_t$. With this notion of a harmonic function we can state the elliptic Harnack inequality  which is a statement about the behavior of harmonic functions.

\begin{definition}\label{def:ehi}
The process $X$ satisfies the elliptic Harnack inequality (EHI) if there exists a constant $c_1$ such that the following holds: for any ball $B(x,R)$, whenever $u$ is a non-negative harmonic function on $B(x,R)$ then there is a quasi-continuous modification, $\tilde{u}$, of $u$ that satisfies
\begin{equation}
	\sup_{B(x,R/2)} \tilde{u} \le c_1 \inf_{B(x,R/2)} \tilde{u}.
\end{equation}
\end{definition}

We need a few intermediate results before we can prove that the EHI is satisfied by these invariant processes. 

\begin{lemma}
Let $\C{E} \in \F{C}$, $r \in (0,1)$, and $h$ be bounded and harmonic in $B=B(x_0,r)$. Then there exists $\theta>0$ such that 
\begin{equation}
	|h(x)-h(y)| \le C \left( \frac{|x-y|}{r} \right)^{\theta} \left( \sup_B |h| \right), \hspace{.5cm} x,y \in B(x_0,r/2), x \sim_m y.
\end{equation}
\end{lemma}

\begin{proof}
This follows from Theorem \ref{thm:coupling} by the same argument as Theorem 4.2 in \cite{BarlowBass1999}.
\end{proof}

\begin{prop}\label{prop:holder}
Let $\C{E} \in \F{C}$ and $h$ be bounded and harmonic in $B(x_0,r)$. Then there exists a set $\C{N}$ of $\C{E}-$capacity 0 such that 
\begin{equation}
	|h(x)-h(y)| \le C \left( \frac{|x-y|}{r} \right)^{\theta} \left( \sup_B |h| \right), \hspace{.5cm} x,y \in B(x_0,r/2) \setminus \C{N}
\end{equation}
\end{prop}

\begin{proof} The proof of Proposition 4.20 in \cite{BBKT2010} carries over to this situation unchanged. 
\end{proof}

\begin{lemma}\label{lem:hitscaling}
Let $\C{E} \in \F{C}$. Then there exist constants $\kappa, 0< C_1$ depending only on $L$ such that if $0 < r<1, x \in L, y \in B(x,C_1r)$ then for all $0 < \delta < C_1$,
\begin{equation}
	\B{P}^{y}(T_{B(x,\delta r)} < \tau_{B(z,r)}) > \delta^{\kappa}.
\end{equation}
\end{lemma}

\begin{proof} This follows from the cell structure by the same argument as Corollary 3.24 in \cite{BarlowBass1999}. Again the only difference is that Laakso spaces only have slide moves and no corner moves. The constant $\kappa$ measures the difference in ``depth'' of $B(x,\delta r)$ and $B(z,r)$. That is if $r$ is between $d_m^{-1}$ and $d_{m+1}^{-1}$ while $\delta r$ is between $d_{m+4}^{-1}$ and $d_{m+5}^{-1}$, then $\kappa = 3$. Also $C_1 < 1/2$. 
\end{proof}

\begin{theorem}\label{thm:ehi}
The Markov processes associated to the Dirichlet forms in \F{C} satisfy the elliptic Harnack inequality. That is for any $z \in L$ and $h$ non-negative and harmonic on $B(z,r)$ that a quasi-continuous modification $\tilde{h}$ satisfies
\begin{equation}
	\tilde{h}(x) \le c \tilde{h}(y), \hspace{1cm} x,y \in B(z,r/2) \setminus \C{N}
\end{equation}
for some constant $c$ that depends only on $L$ and \C{E}-capacity zero set \C{N}.
\end{theorem}

\begin{proof} 
By looking at $h+\epsilon$ and then letting $\epsilon \downarrow 0$, we may assume that $h$ is bounded below by a positive constant in $B(z,r)$. Multiplying by a constant, we may assume that $\inf_{B(z,r/2)} h = 1$. By Proposition \ref{prop:holder} we have that $h$ is bounded, positive, and quasi-continuous in $B(z,r/2)$ so it is it's own quasi-continuous modification. 

By Lemma \ref{lem:hitscaling} we have for $x,y \in B(z,C_1r)$ that
\begin{equation}
	\B{P}^{y}(T_{B(x,\delta r)} < \tau_{B(z,r)}) > \delta^{\kappa}.
\end{equation}
For some $\delta \in (0,C_1)$. Which gives the estimate of the minimum of $h$ on 
\begin{equation}
	1 = h(y) \ge \B{E}^{y}[h(X(T_{B(x,\delta r)}));T_{B(x,\delta r)} < \tau_{B(z,r)}] \ge \delta^{\kappa} \inf_{B(x,\delta r)} h
\end{equation}
so that
\begin{equation}
	\inf_{B(x,\delta r)} h \le \delta^{-\kappa} \hspace{1cm} x \in B(z,\delta r).
\end{equation}

 We then follow the argument in the proof of Theorem 4.3 in \cite{BBKT2010} to to control the oscillation of harmonic functions using the H\"older continuity from above in terms of a power of $\delta$ that depends only on $L$. 
\end{proof}

\section{Heat Kernel Estimates}\label{sec:hkest}

While waiting for \cite{GrigorTelcs2011} to appear as a preprint the authors of \cite{BBKT2010} wrote a set of supplementary notes \cite{BBKT2008b}. In the these notes conditions equivalent to two-sided Gaussian estimates on a heat kernel are given. More detailed accounts are of course to be found in \cite{GrigorTelcs2011}. In \cite{BBKT2010} one set of equivalent conditions are used, but we shall use the second. We define the conditions in the theorem before stating the theorem itself.

\begin{definition}\label{def:vd}
A metric measure space $(X,d,\mu)$ has the \emph{volume doubling} (VD) property if there exists a constant $C_1$ such that 
\begin{equation}
	\mu(B(x,2R)) \le C_1 \mu(B(x,R) \hspace{1cm} \forall x \in X, 0 \le R \le 1.
\end{equation}
Where $B(x,R)$ is the metric ball of radius $R$ centered at $x$.
\end{definition}

We will have need of a function with two useful properties time doubling and fast time growth. Let $H(r):[0,2] \rightarrow [0,\infty)$ be strictly increasing and $H(1) \in [C_2,C_3]$ where $C_2,C_3$ are positive constants. Then $H(r)$ has the \emph{time doubling} property if there exists a positive constant, $C_4$ such that
\begin{equation}
	H(2R) \le C_4H(R) \hspace{1cm} \forall R \in (0,1]. \label{eqn:h2}
\end{equation}
The function $H(r)$ has the \emph{fast time growth} property if there exists $C_5 >0$ and $\beta_1 >1$ such that 
\begin{equation}
	\frac{H(R)}{H(r)} \ge C_5 \left( \frac{R}{r} \right)^{\beta_1}.
\end{equation}
If $\beta_2 = \log(C_4)/\log(2)$ then we have two sided estimates on $H(R)/H(r)$ of the form:
\begin{equation}
	C_5 \left( \frac{R}{r} \right)^{\beta_1} \le \frac{H(R)}{H(r)} \le C_7 \left( \frac{R}{r} \right)^{\beta_2}. \label{eqn:h1}
\end{equation}

\begin{definition}\label{def:timescale}
A function $H(r)$ with these properties is called a \emph{time scaling function} $(H,\beta_1,\beta_2,)$.
\end{definition}

A pertinent example of a time scaling function $(H,\beta,\beta)$ is $H(r) = r^{\beta}$.

\begin{definition}\label{def:resist}
Let $(X,d,\mu,(\C{D},\C{F}))$ be a metric measure Dirichlet space and $A,B$ be disjoint subsets of $X$. The \emph{effective resistance}, $R_{\C{E}}(A,B)$, is given by 
\begin{equation}
	R_{\C{E}}(A,B)^{-1} = \inf \left\{ \C{E}(f,f) : f=0\ on\ A\ and\ f=1\ on\ B , f \in \C{F} \right\}.
\end{equation}
We say that $X$ satisfies the condition $RES(H)$ if there exists constants $c_1,c_2$ and a time scaling function, $H(r)$, such that for any $x_0 \in X, 0 \le R\le \frac{1}{3}$,
\begin{equation}
	c_1 \frac{H(R)}{\mu(B(x_0,R))} \le R_{\C{E}}\left(B(x_0,R), B(x_0,2R)^{C} \right) \le c_2\frac{H(R)}{\mu(B(x_0,R))}.
\end{equation}
\end{definition}

\begin{definition}\label{def:exit}
Let $H(r)$ be a time scaling function and $\tau_{B(x_0,R)}$ is the first time that $X_t$ exists the metric ball $B(x_0,R)$. Then if for $0 < R \le \frac{1}{3}$ and any $x_0 \in L$ the estimate
\begin{equation}
	c_1H(R) \le \B{E}^{x_0} \left[ \tau_{B(x_0,R)} \right] \le c_2H(R)
\end{equation}
Holds then we say that $X_t$ satisfies the exit time condition $(E(H))$.
\end{definition}

\begin{lemma}\label{lem:exittime}
The strong Markov processes associated to Dirichlet forms in \F{C} satisfy the exit time condition $(E(H))$. 
\end{lemma}

\begin{proof} Let $X_t$ be a Markov process corresponding to $\C{E} \in \F{C}$. Given $x_0 \in L$ and $R < 2 dist(x_0, \partial L)$ then exists a maximal $N \ge 0$ such that there exists an $N-$cell containing $x_0$ and entirely contained in $B(x_0,R)$ called $S_1$. Then $X_t$ starting at $x_0$ a.s. leaves this cell through its boundary, so 
\begin{equation}
	\B{E}^{x_0}[\tau_{B(x_0,R)}] \ge \B{E}^{x_0}[\tau_{S_1}].
\end{equation}
However $\B{E}^{x_0}[\tau_{S_1}]$ is equal to the exit time to $\pi(X_t)$ where $\pi$ is the projection of $S_1$ onto a line interval with length equal to the diameter of $S_1$. Similarly there is a minimal $M$ such that a $M$-cell, $S_2$,  containing $B(x_0,R)$ where the analogous upper bound can be formulated. Note that $diam(S_2) \le (j+1)S_1$ so that the lengths of these two cells is comparable. 

This reduces the problem to exit time estimates on intervals of the real line. Since the Dirichlet forms in \F{C} are local, regular, conservative, and symmetric so are their projections onto the real line. If the results of \cite{BBKT2010} are applied the the generalized Sierpinski carpet that is $[0,1]^{2}$ the heat kernel bounds are Gaussian. By Theorem \ref{thm:hkequiv} exit time estimates on the square are controled by $r^{2}$. If the processes associated with locally symmetric Dirichlet forms are projected onto $[0,1]$ we have the same class of processes as are obtained from projecting the processes associated to elements of \F{C} to $[0,1]$. Thus the exit times are controled by $H(r) =r^{2}$.
\end{proof}

\begin{definition}\label{def:hkgauss}
Let $H(r)$ be a time scaling function with inverse function $h(t)$. Then if there exists constants $c_0,\beta_1,\beta_2$ such that
\begin{equation}
 	\frac{1}{c_0 \mu(B(x,h(t)))}e^{-c_0(\frac{H(d(x,y))}{t})^{\frac{1}{\beta_1-1}}} \le p_t(x,y) \le \frac{c_0}{ \mu(B(x,h(t)))}e^{-c^{-1}_0(\frac{H(d(x,y))}{t})^{\frac{1}{\beta_2-1}}}
\end{equation}
Then the heat kernel $p_t(x,y)$ satisfies the condition $HK(H,\beta_1,\beta_2,c_0)$.
\end{definition}

Recall also the definition of the Elliptic Harnack Inequality from Definition \ref{def:ehi}. 

\begin{theorem}\label{thm:hkequiv}
Let $X,d,\mu,(\C{E},\C{F}))$ be a metric measure Dirichlet space with $d$ a geodesic metric and $(\C{E},\C{F})$ conservative. Let $H(r)$ be a time scaling function with constants $C_1 ... C_7, \beta_1, \beta_2$. Then the following statements are all equivalent.
\begin{enumerate}
	\item The space $X$ satisfies volume doubling, elliptic Harnack inequality, and $RES(H)$.
	\item The space $X$ satisfies volume doubling, elliptic Harnack inequality, and exit time estimates $(E(H))$.
	\item The space $X$ satisfies $HK(H,\beta_1,\beta_2,c_0)$.
\end{enumerate}
\end{theorem}

This is Theorem 1.3 in \cite{BBKT2008b} and a major result of \cite{GrigorTelcs2011}. In \cite{BBKT2010} the $1 \Rightarrow 3$ is used instead because the resistance estimates had been well studied in previous work by these authors and others, \cite{BarlowBass1990,McGillivray2002}. The proofs of these implications are \emph{effective}, that is from the information in Theorem \ref{thm:hkequiv}.2 the constants $\beta_1,\beta_2$, and $c_0$ as well as the function $H$ could be given explicitly.

\begin{theorem}\label{thm:hk}
On any given Laakso space, the Laplacians associated to the Dirichlet forms in \F{C} all satisfy $HK(H,2,2,c_0)$, where $c_0$ depends only on $L$.
\end{theorem}

\begin{proof} We using Theorem \ref{thm:hkequiv} to reduce this proof to checking three conditions, volume doubling, elliptic Harnack inequality and exit time estimates. Volume doubling is a consequence of the $Q-$Ahlfors regularity of Laakso spaces as given in \cite{Laakso2000}. The elliptic Harnack inequality was shown in Theorem \ref{thm:ehi}. The exit time estimates were shown in Lemma \ref{lem:exittime}.
\end{proof}

Following \cite[Theorem 4.1]{KumagaiSturm2005} the domain of an arbitrary $\C{E} \in \F{C}$ is characterized as a Besov-Lipschitz space. Let
\begin{eqnarray}
	J_r(f) &=& r^{-\alpha} \int_L \int_{B(x,r)} |f(x) - f(y)|^{2}\ d\mu(y)\ d\mu(x),\\
	N^{r}_H(f) &=& H(r)^{-1}J_r(f),\\
	N_H(f) &=& \sup_{0<r\le1} N_H^{r}(f),\\
	W_H &=& \{ u \in L^{2}(L,\mu_{\infty}) : N_H(f) < \infty \}.
\end{eqnarray}

\begin{theorem}\label{thm:equivdomain}
Let $H$ satisfy (\ref{eqn:h1}) and (\ref{eqn:h2}). Suppose $p_t$ satisfies $HK(H,\beta_1,\beta_2,c_0)$. Then 
\begin{equation}
	C_1\C{E}(f,f) \le \limsup_{j \rightarrow \infty} N_H^{r_j}(f) \le N_H(f) \le C_2\C{E}(f,f) \hspace{1cm} f \in W_H
\end{equation}
where the constants $C_i$ depend only on the constants in (\ref{eqn:h1}) and (\ref{eqn:h2}), and in $HK(H,\beta_1,\beta_2,c_0)$. Further
\begin{equation}
	\C{F} = W_H.
\end{equation}
\end{theorem}

The possibility that \F{C} can contain two families of Dirichlets forms with different domains is now excluded. See the comments after Proposition \ref{prop:closedcombo}.

\section{Uniqueness of Brownian Motion}\label{sec:uniq}
Following on Theorem \ref{thm:equivdomain} where it was shown that the elements of \F{C} have a common domain the following definition of the Hilbert projective metric is well-defined. Then Theorem \ref{thm:hilbertbound} shows that in the projective metric \F{C} is a bounded set. Finally Theorem \ref{thm:onedim} shows that \F{C} must them be a one dimensional vector space. 

\begin{definition}\label{def:hilbert}
Let $W = W_{H}$ be as defined above. Let $\C{A},\C{B} \in \F{C}$. We say that $\C{A} \le \C{B}$ if 
\begin{equation}
	\C{B}(u,u) - \C{A}(u,u) \ge 0\ for\ all\ u \in W.
\end{equation}
For $\C{A},\C{B} \in \F{C}$ define
\begin{eqnarray}
	\sup(\C{B}|\C{A}) &=& \sup\left\{\frac{\C{B}(f,f)}{\C{A}(f,f)} : f \in W \right\}\\
	\inf(\C{B}|\C{A}) &=& \inf\left\{\frac{\C{B}(f,f)}{\C{A}(f,f)} : f \in W \right\}\\
	h(\C{A},\C{B}) &=& \log\left(\frac{\sup(\C{B}|\C{A})}{\inf(\C{B}|\C{A})} \right);
\end{eqnarray}
$h$ is Hilbert's projective metric and we have $h(\theta\C{A},\C{B}) = h(\C{A},\C{B})$ for any $\theta \in (0,\infty)$. Note that $h(\C{A},\C{B})=0$ if and only if \C{A} is a nonzero constant multiple of \C{B}.
\end{definition}

\begin{theorem}\label{thm:hilbertbound}
There exists a constant $C_L$, depending only on $L$, such that if $\C{A}.\C{B} \in \F{C}$ then
\begin{equation}
	h(\C{A},\C{B}) \le C_L.
\end{equation}
\end{theorem}

\begin{proof} By Theorem \ref{thm:equivdomain} we have positive, finite constants such that for any $f \in W$
\begin{equation}
	\frac{C_1}{C_2} \le \frac{\C{B}(f,f)}{\C{A}(f,f)} \le \frac{C_2}{C_1}=C_L.
\end{equation}
Which gives crude bounds for $\sup(\C{B}|\C{A})$ and $\inf(\C{B}|\C{A})$ of $C_2/C_1$ and $C_1/C_2$ respetively. Thus $h(\C{A},\C{B}) \le 2\log(C_2/C1)$.
\end{proof}

The following is the main theorem of the paper. It is proven by the same means as Theorem 1.2 of \cite{BBKT2010}.

\begin{theorem}\label{thm:onedim}
The family of $L-$invariant Dirichlet forms, \F{C}, is one dimensional. 
\end{theorem}

\begin{proof} By Lemma \ref{lem:DFinC} \F{C} is non-empty. 

Let $\C{A},\C{B} \in \F{C}$ and $\lambda = \inf(\C{B}|\C{A})$. Let $\delta >0$ and $\C{C} = (1+\delta)\C{B} - \lambda\C{A}$. By Theorem \ref{thm:combo} \C{C} is a local regular Dirichlet form on $L^{2}(L,\mu_{\infty})$ and $\C{C} \in \F{C}$.  Since
\begin{equation}
	\frac{\C{C}(f,f)}{\C{A}(f,f)} = (1+\delta) \frac{\C{B}(f,f)}{\C{A}(f,f)} - \lambda
\end{equation}
for $f \in W$, we obtain
\begin{equation}
	\sup(\C{C}|\C{A}) = (1+\delta)\sup(\C{B}|\C{A}) - \lambda,
\end{equation}
and
\begin{equation}
	\inf(\C{C}|\C{A}) = (1+\delta)\inf(\C{B}|\C{A}) - \lambda = \delta \lambda.
\end{equation}
Hence for any $\delta>0$,
\begin{equation}
	e^{h(\C{A},\C{C})} = \frac{(1+\delta)\sup(\C{B}|\C{A}) - \lambda}{\delta\lambda} \ge \frac{1}{\delta}(e^{h(\C{A},\C{B})}-1).
\end{equation}
If $h(\C{A},\C{B})>0$, this is not bounded as $\delta \rightarrow 0$, contradicting Theorem \ref{thm:hilbertbound}. Therefore $h(\C{A},\C{B}) = 0.$
\end{proof}

\begin{cor}
If $\C{E} \in \F{C}$, then there exists a non-zero constant $c$ such that $\C{E} = c\C{E}_M$.
\end{cor}

\section{Minimal Generalized Upper Gradient Laplacian}\label{sec:MGUG-Lap}
In \cite{Laakso2000} the existence of minimal generalized upper gradients is used to show that there is a weak $(1,1)-$Poincar\'e inequality on Laakso spaces. Cheeger, in \cite{Cheeger1999} had already shown how minimal generalized upper gradients can be used in a version of Sobolev theory as replacements for derivatives. Here we use them to exhibit an $L-$invariant Dirichlet form constructed in \cite{Steinhurst2010} and shown to be $L-$invariant in this paper. 

This Laplacian and its spectrum were analyzed in \cite{RomeoSteinhurst2009,BegueEtAl2009}. In those papers the spectrum was computed along with the dimension of each eigenspace. The Hausdorff, walk, and spectral dimensions were computed for every Laakso space with $k=1$ in \cite{BegueEtAl2009}.

\begin{definition}\label{def:mgug}
In a metric measure space, $(X,d,\mu)$, a generalized upper gradient of a continuous function $f \in C(X)$ is a function $p_f$ taking values in $[0,\infty]$ with the following property:
\begin{equation}
	\left| f(x) - f(y) \right| \le \int_{\gamma} p_f\ dm
\end{equation}
Where $x,y \in X$ and $\gamma:[0,1] \rightarrow X$ is any rectifiable path from $x$ to $y$, and $dm$ is the measure induced on the image of $\gamma$ by the arc length parameterization of $\gamma$. Such a function $p_f$ exists for any $f \in C(X)$ because $p_f = \infty$ is a generalized upper gradient for every continuous function.

A minimal generalized upper gradient (MGUG) is a generalized upper gradient that is $\mu-$almost everywhere less than or equal to all other generalized upper gradients. 
\end{definition}

Many of the basic properties of minimal generalized upper gradients, including the fact that the space of square integrable functions with square integrable minimal generalized upper gradients and is a closable space is shown in \cite{Cheeger1999}. The following proposition is one of the main resultes of \cite{Steinhurst2010}.

\begin{prop}\label{prop:mgugDF}
There exists a local, regular Dirichlet form, $(\C{E}_M,\C{F}_M)$, such that for $f \in \C{F}_M$
\begin{equation}
	\C{E}_M(f,f) = \int_L p_f^{2}\ d\mu.
\end{equation}
This Dirichlet form will be refered to as the \emph{MGUG Dirichlet form} when it needs to be distinguished.
\end{prop}

\begin{lemma}\label{lem:DFinC}
The Dirichlet form $(\C{E}_M, \C{F}_M)$ is in $\F{C}$ (see Definition \ref{def:invar}).
\end{lemma}

\begin{proof}
We first describing a dense subset of $\C{F}_M$ on which invariance can be checked and then we claim by density that $\C{E}_M$ is $L-$invariant on its entire domain. Denote by $\C{D}_n$ those functions on $L$ that are continuous, for given $w \in L^{k}$ are piecewise differentiable on the line segment $\iota(I \times w)$, and for given $x \in I$ is constant on each $n'th$ level cell of $\iota(x \times K^{k}) \simeq K^{k}$. In \cite{Steinhurst2010} it is shown that
\begin{equation}
	\overline{ \bigcup_{n=0}^{\infty} \C{D}_n } = \C{F}_M.
\end{equation}

Fix an $S \in \C{S}_n$ and $f \in \C{D}_m$, then $R_S f$ is a function on the cell $S$ and then when $U_S$ is applied to unfold $R_Sf$ to the entire space we get that it is continuous, piecewise differentiable on each $\iota(I \times w)$ for all $w \in K^{k}$, and constant on $m^{th}$ level cells of $\iota( x \times K^{k})$ for all $x \in I$ if $m > n$ and constant on $n^{th}$ level cells if $n \ge m$. This means that $U_SR_S f \in \C{D}_m$. Since both $R_S$ and $U_S$ have operator norm less than one and have domains containing all of $\C{F}_M$ they are continuous operators on $\C{F}_M$ so the invariance extends to all of $\C{F}_M$. This is part 1 of Definition \ref{def:invar}.

We next check part 2 of Definition \ref{def:invar}. Let $S_1,S_2 \in \C{S}_n$ for some $n \ge 0$ and $\tilde{\Phi} = \iota \circ \Phi \circ \iota^{-1}$ and isometry mapping $S_1$ onto $S_2$. If $f \in \C{F}^{S_2}$ the we claim that $p_{f \circ \tilde{\Phi}} = p_f \circ \tilde{\Phi}$. Given this claim and the definition of $\C{E}_M$ as the integral of $p_f^{2}$ we then have 
\begin{equation}
	\C{E}_M^{S_1}(f \circ \tilde{\Phi},f \circ \tilde{\Phi}) = \C{E}_m^{S_2}(f,f).
\end{equation}
We check the claim on $\C{D}_m$ for arbitrary $m \ge 0$, the claim extending to all of $\C{F}_M$ by density. On $\C{D}_m$ we can view $f \circ \iota$ as a piece-wise differentiable function in the $I$ coordinate and piecewise constant in the $K^{k}$ coordinate. Since $\Phi$ is an isometry of $\B{R}^{1+k}$ that maps $\iota^{-1}(S_1)$ onto $\iota^{-1}(S_2)$ it must either be a composition of reflections and translations. In the case of $\Phi$ being a translation then the claim holds because differentiation commutes with translation. In the case of $\Phi$ being a reflection differentiation is anti-commutative, but since the minimal generalized upper gradient is $p_f = \left| \frac{\partial f}{\partial x} \right|$ the claim holds. Then in the case of composition of translation and reflection the claim holds.

Part 3 of Definition \ref{def:invar} holds because $\C{E}_M$ is given by an integral with respect to a measure that givens no mass to the intersection of cells.
\end{proof}

\small
\bibliography{heatkernelbib}{}
\bibliographystyle{plain}

\end{document}